\documentclass[3p,12pt]{elsarticle}
\usepackage{mathtools,commath}
\usepackage{amssymb,amsmath,amsthm}
\usepackage{enumitem}
\usepackage{bbm}
\usepackage{hyperref}
\usepackage{fullpage}
\usepackage{bigints}
\usepackage{subfiles}

\newtheorem{theorem}{Theorem}
\newtheorem{lemma}[theorem]{Lemma}

\newtheorem{remark}[theorem]{Remark}

\numberwithin{equation}{section}
\numberwithin{theorem}{section}
\renewenvironment{proof}{\hspace*{-0.6cm}{\bfseries Proof:}}{\hfill $\blacksquare$}
\def\bR{\mathbb R}

\begin{document}
\begin{frontmatter}

\title{Existence of Positive Radial Solutions of General Quasilinear Elliptic Systems}

\author{Daniel Devine}
\ead{dadevine@tcd.ie}
\address{School of Mathematics, Trinity College Dublin, Dublin 2, Ireland}

\begin{abstract}
Let $\Omega\subset\bR^{n}\ (n\geq2)$ be either an open ball $B_R$ centred at the origin or the whole space. We study the existence of positive, radial solutions of quasilinear elliptic systems of the form 
\begin{equation*}
\left\{
\begin{aligned}
\Delta_{p} u&=f_1(|x|)g_1(v)|\nabla u|^{\alpha}  &&\quad\mbox{ in } \Omega, \\
\Delta_{p} v&=f_2(|x|)g_2(v)h(|\nabla u|) &&\quad\mbox{ in } \Omega,
\end{aligned}
\right.
\end{equation*}
where $\alpha\geq 0$, $\Delta_{p}$ is the $p$-Laplace operator, $p>1$, and for $i,j=1,2$ we assume $f_i,g_j,h$ are continuous, non-negative and non-decreasing functions. For functions $g_j$ which grow polynomially, we prove sharp conditions for the existence of positive radial solutions which blow up at $\partial B_{R}$, and for the existence of global solutions. 
\end{abstract}

\end{frontmatter}

\section{Introduction}
Let $\Omega\subset\bR^{n}\ (n\geq2)$ be either an open ball $B_R$ centred at the origin or the whole space, and for $i,j=1,2$ suppose $f_i,g_j,h$ are continuous, non-negative and non-decreasing functions. We study the existence of positive, radial solutions of quasilinear elliptic systems of the form 
\begin{equation}\label{system}
\left\{
\begin{aligned}
\Delta_{p} u&=f_1(|x|)g_1(v)|\nabla u|^{\alpha}   &&\quad\mbox{ in } \Omega, \\
\Delta_{p} v&=f_2(|x|)g_2(v)h(|\nabla u|) &&\quad\mbox{ in } \Omega,
\end{aligned}
\right.
\end{equation}
 where $\Delta_{p}u=\mbox{div}(|\nabla u|^{p-2}\nabla u)$ is the $p$-Laplace operator, $p>1$ and $\alpha\geq 0$. For $\Omega=B_R$, the following boundary behaviours are possible:
\begin{enumerate}[label=(B\arabic*)]
    \item $u$ and $v$ are bounded in $B_{R}$;
    \item $u$ is bounded in $B_{R}$ and $\lim\limits_{|x|\to R^{-}}v(x)=\infty$;
    \item $\lim\limits_{|x|\to R^{-}}u(x)=\lim\limits_{|x|\to R^{-}}v(x)=\infty$.
\end{enumerate}
The first equation of \eqref{system} prevents a positive radial solution satisfying $\lim_{|x|\to R^{-}}u(x)=\infty$ and $v$ is bounded in $B_R$ from existing (see Lemma \ref{no u2inf} below for details). We prove sharp conditions for the existence of positive radial solutions of \eqref{system} with boundary behaviours (B1)-(B3). We then turn our attention to the case $\Omega=\bR^{n}$, and show that global positive radial solutions of \eqref{system} exist if and only if all positive radial solutions satisfy (B1). 

The study of elliptic systems without the presence of gradient terms has a rich history. For semilinear elliptic systems we refer to, for instance, \cite{BVP2001, CR2002, GMRZ2009, L2011, LW1999, LW2000}, and for the case of quasilinear elliptic systems we refer to \cite{ACM2002, BVG2010, BVG1999, BV2000, BVP2001, CFMT2000,F2011}. More recently, systems with gradient terms such as \eqref{system} have been investigated \cite{BGW22, BVG23, DLS2005, F2013,FV17,GGS19,S15}, which we shall now discuss.  
 
Filippucci and Vinti \cite{FV17} studied the positive radial solutions of the quasilinear elliptic system 
\begin{equation}\label{FVsystem}
\left\{
\begin{aligned}
\Delta_{p} u&=v^m  &&\quad\mbox{ in } B_R, \\
\Delta_{p} v&=h(|\nabla u|) &&\quad\mbox{ in } B_R,
\end{aligned}
\right.
\end{equation}
where $m>0$ and $h\in C^{1}([0,\infty))$ is non-negative and non-decreasing. In the case $p\geq 2$, the authors show that a necessary condition for system \eqref{FVsystem} to admit a positive radial solution $(u,v)$
which blows up at $\partial B_R$ is given by
\begin{equation}\label{FVcond}
      \int_{1}^{\infty}\frac{s^{(p-2)(p-1)/(mp+p-1)} ds}{\left(\int_{0}^{s}\sqrt[\leftroot{-3}\uproot{3}p]{h(t)}dt\right)^{mp/(mp+p-1)}}<\infty.
\end{equation}

\noindent This extended an earlier result due to Singh \cite{S15}, who studied system \eqref{FVsystem} in the semilinear case $p=2$. The results of \cite{FV17, S15} were later built upon by Ghergu, Giacomoni and Singh \cite{GGS19} in the particular case $h(t)=t^q$, for $q>0$. The authors fully classified the positive radial solutions of the system
\begin{equation}\label{GGSsystem}
\left\{
\begin{aligned}
\Delta_{p} u&=v^m|\nabla u|^{\alpha}  &&\quad\mbox{ in } B_R, \\
\Delta_{p} v&=v^{\beta}|\nabla u|^{q} &&\quad\mbox{ in } B_R,
\end{aligned}
\right.
\end{equation}
according to their behaviour at $\partial B_{R}$. The exponents in \eqref{GGSsystem} are assumed to satisfy $m,q>0$, $\alpha\geq 0$, $0\leq \beta\leq m$ and $(p-1-\alpha)(p-1-\beta)-qm\neq 0$. Optimal conditions on the exponents for the existence of positive radial solutions with boundary behaviours (B1)-(B3) are found, and the exact rate of blow-up of solutions of \eqref{GGSsystem} was later investigated in \cite{BGW22}. 

Inspired by the above works, our aim in this paper is to extend the results of \cite{FV17}, as well as the existence results of \cite{GGS19}, to the much more general class of quasilinear elliptic systems \eqref{system}. Throughout this paper, we assume for each $i,j$ that
\begin{itemize}
    \item[(A1)] $f_i,g_j,h$ are continuous, non-decreasing on $[0,\infty)$ and positive on $(0,\infty)$.
\end{itemize}
When it comes to the functions $g_j$, we shall make one more assumption, namely that they grow polynomially. In particular, we shall also assume
\begin{itemize}
    \item [(A2)] there exist constants $k_1> 0$ and $ 0\leq k_2\leq k_1$ such that $g_{j}(t)/t^{k_j}$ is non-increasing in $(0,\infty)$ and 
\begin{equation*}
    \lim\limits_{t\to\infty}\frac{g_{j}(t)}{t^{k_j}}>0.
\end{equation*}
\end{itemize}
Systems \eqref{FVsystem} and \eqref{GGSsystem} clearly satisfy (A1)-(A2), but our assumption (A2) allows a much more general class of functions $g_j$ to be considered.  For example, each $g_j(t)$ could be any sum of non-negative powers of $t$ with non-negative coefficients. When it comes to the weight functions $f_i(t)$, once they satisfy (A1) it turns out they do not play a significant role in the analysis, so they can be quite arbitrary. We shall see that this essentially comes from the fact that their continuity on $[0,\infty)$ guarantees $\sup_{x\in B_R}f_i(|x|)=f_i(R)$ is finite for any $R>0$, so they do not influence the behaviour of a solution $(u,v)$ near $\partial B_R$. The main contribution of this paper, though, is the sharpness of the results for general functions $h(t)$ which are only assumed to satisfy (A1).

Our first goal in this paper is to give sharp conditions for the existence of non-constant positive radial solutions of \eqref{system} in $\Omega=B_R$ which satisfy (B1)-(B3). According to \cite[Theorem 1.2]{FV17}, condition \eqref{FVcond} is necessary for positive radial solutions of system \eqref{FVsystem} to satisfy (B2) or (B3). We extend this result in three directions. Namely, we extend it to the much more general system \eqref{system}, we extend it to all $p>1$, and we also show that our conditions are sharp. In fact, our result also extends \cite[Theorem 1.1]{FV17} and \cite[Theorem 2.1]{GGS19}. In order to prove the sharpness of these conditions, we first need to prove an integral comparison result, which is given at the end of Section \ref{prelim}.

The second goal of this paper is to give sharp conditions for the existence of global non-constant positive radial solutions of \eqref{system}. Using a simple comparison argument, we show that global positive radial solutions exist if and only if all positive radial solutions satisfy (B1). In other words, either all positive radial solutions of \eqref{system} are global or none are.

Before stating the main results of this paper, we again clarify that we are only interested in positive radial solutions, by which we shall always mean a pair $(u(r),v(r)), r=|x|$,  such that 
\begin{itemize}
    \item $u,v\in C^{2}(\Omega)$ are positive and solve \eqref{system};
    \item $u$ and $v$ are non-constant in any neighbourhood of the origin.
\end{itemize}
In the case $\Omega=\bR^{n}$, we call solutions of \eqref{system} global solutions.

\begin{theorem}\label{quas case}
Assume (A1),(A2), $0\leq\alpha<p-1$, and let $\theta=\frac{1}{p-1-\alpha}$. Let $(u,v)$ be a non-constant positive radial solution of \eqref{system}. Then:
\begin{enumerate}
    \item $(u,v)$ satisfies (B1) if and only if
    \begin{equation}\label{u v bound}
        \int_{1}^{\infty}\frac{ds}{\left(\int_{0}^{s}h(t^{\theta})^{\frac{1}{p}}dt\right)^{\frac{k_1p}{k_1p+p-1-k_2}}}=\infty
    \end{equation}
     \item $(u,v)$ satisfies (B2) if and only if
    \begin{equation}\label{v inf u bound}
        \int_{1}^{\infty}\frac{s^{\theta}\ ds}{\left(\int_{0}^{s}h(t^{\theta})^{\frac{1}{p}}dt\right)^{\frac{k_1p}{k_1p+p-1-k_2}}}<\infty
    \end{equation}
    \item $(u,v)$ satisfies (B3) if and only if
    \begin{equation}\label{u v inf}
                \int_{1}^{\infty}\frac{ds}{\left(\int_{0}^{s}h(t^{\theta})^{\frac{1}{p}}dt\right)^{\frac{k_1p}{k_1p+p-1-k_2}}}<\infty\quad \mbox{and}\quad \int_{1}^{\infty}\frac{s^{\theta}\ ds}{\left(\int_{0}^{s}h(t^{\theta})^{\frac{1}{p}}dt\right)^{\frac{k_1p}{k_1p+p-1-k_2}}}=\infty.
    \end{equation}
\end{enumerate}
\end{theorem}

\begin{theorem}\label{global}
    Assume (A1),(A2), $\Omega=\bR^{n}$ and $\alpha\geq 0$. Then \eqref{system} admits non-constant global positive radial solutions if and only if $\alpha<p-1$ and \eqref{u v bound} holds.
\end{theorem}

\begin{remark}
\normalfont{
Although not explicitly treated here, the results in this paper actually hold for several generalisations of system \eqref{system}. For example, in the first equation we could replace $g_1(v)$ by $g_1(v)\cdot g_3(u)$ for a continuous, non-decreasing and bounded function $g_3$ which is positive on $(0,\infty)$. When it comes to the second equation, we could similarly replace $g_2(v)$ by $g_2(v)\cdot g_4(u)$, with $g_4$ subject to the same conditions as $g_3$. Such additional factors would not affect the arguments whatsoever. For the sake of simplicity, though, such variations are not treated explicitly here. 
}
\end{remark}

The remaining sections are organized as follows. In Section \ref{prelim} we gather some preliminary information about solutions of system \eqref{system}. In particular, we show all positive radial solutions of \eqref{system} are increasing and convex. In Section \ref{proof thm quas} we prove Theorem \ref{quas case}, and finally the proof of Theorem \ref{global} is given in Section \ref{proof thm pol}.

\section{Properties of Solutions}\label{prelim}
In this section, we gather some properties which all non-constant positive radial solutions of \eqref{system} must have. We first show that if $(u,v)$ is a positive radial solution, then both $u$ and $v$ are increasing and convex. This explains why solutions which exist locally, but not globally,  must blow up around $\partial\Omega$. We also show that if \eqref{system} admits any positive radial solutions, then we necessarily have $0\leq \alpha <p-1$. The proofs of these results are similar to those in \cite{GGS19,S15}, but the necessary modifications for the more general system \eqref{system} are included for completeness. We conclude the section with an integral comparison result which will allow us to prove the sharpness of conditions \eqref{u v bound}-\eqref{u v inf}. 
\begin{lemma}\label{monotonic}
    Assume (A1), $\Omega=B_{R}$, $\alpha\geq 0$, and suppose $(u,v)$ is a non-constant radial solution of \eqref{system} with $u(0)>0$ and $v(0)>0$. Then $u'(r),v'(r)>0$ for all $0<r<R$.
\end{lemma}

\begin{proof}
   Since $(u,v)$ is a radial solution of \eqref{system}, we can rewrite \eqref{system} as
 \begin{equation}\label{sysrad}
	\left\{
	\begin{aligned}
	&[r^{n-1}u'(r)|u'(r)|^{p-2}]'=r^{n-1}f_1(r)g_1(v(r))|u'(r)|^{\alpha}&&\quad\mbox{for all}\quad 0<r<R,\\
    &[r^{n-1}v'(r)|v'(r)|^{p-2}]'=r^{n-1}f_2(r)g_2(v(r))h(|u'(r)|)&&\quad\mbox{for all}\quad 0<r<R,\\
    &u'(0)=v'(0)=0, u(0)> 0, v(0)>0. 
	\end{aligned}
	\right.
\end{equation}
Since $v$ is continuous, there exists a maximal subinterval $(0,R_0)\subset (0,R)$ such that $v(r)>0$ for $0< r<R_0$.
By our assumption (A1), we thus have that $r\mapsto  r^{n-1}u'(r)|u'(r)|^{p-2}$ and $r\mapsto  r^{n-1}v'(r)|v'(r)|^{p-2}$ are increasing for $0<r<R_0$, and vanish at $r=0$.
It follows that $u'(r)>0$ and $v'(r)>0$ for all $0<r<R_0$. In particular, $v(r)$ cannot vanish at $r=R_0$, so we in fact have $R_0=R$, implying $v(r)>0$ on $[0,R)$. By the above we have that $u'(r)> 0$ and $v'(r)> 0$ on $(0,R)$. 
\end{proof}

\begin{lemma}\label{alpha<p-1}
 Assume (A1), $\Omega=B_R$ and $\alpha\geq p-1$. Then \eqref{system} does not admit any non-constant positive radial solutions for any $R>0$.
\end{lemma}

\begin{proof}
Suppose a positive radial solution $(u,v)$ exists. From Lemma \ref{monotonic}, we must have $u'(r)>0$ for all $0< r< R$, and so the first equation in \eqref{sysrad} can be expressed as
\begin{equation*}
\frac{\left[ r^{n-1} u'(r)^{p-1} \right]'}{r^{n-1}u'(r)^{p-1}} =
f_1(r) g_1(v(r)) \cdot u'(r)^{\alpha-p+1}
\end{equation*}
for all $0<r<R$.  Integrating this equation over the interval $[0,R/2]$ yields
\begin{equation*}
\ln\left[ r^{n-1} u'(r)^{p-1} \right] \Big|_0^{R/2} =
\int_0^{R/2} f_1(s) g_1(v(s)) \cdot u'(s)^{\alpha-p+1} \,ds.
\end{equation*}
Since $\alpha\geq p-1$ and $f_1, g_1$ are continuous, the right hand side above is clearly finite.  This is
a contradiction because $u'(0)=0$ and $p>1$, so the left hand side is infinite.
\end{proof}

\begin{lemma}\label{convex}
Assume (A1), $\Omega=B_{R}$, $0\leq\alpha<p-1$, and suppose $(u,v)$ is a non-constant positive radial solution of \eqref{system}. Then 
\begin{align}
     \frac{p-1-\alpha}{n(p-1-\alpha)+\alpha}f_1(r)g_1(v) \leq&[u'(r)^{p-1-\alpha}]'\leq \frac{p-1-\alpha}{p-1}f_1(r)g_1(v) ,\label{convex bounds1}\\
         \frac{1}{n}f_2(r)g_2(v)h(u') \leq&[v'(r)^{p-1}]'\leq f_2(r)g_2(v)h(u'),\label{convex bounds2}
\end{align}
for all $0<r<R$. In particular, both $u(r)$ and $v(r)$ are convex for all $0<r<R$.
\end{lemma}
\begin{proof}
From Lemma \ref{monotonic} we see that \eqref{sysrad} can be expressed as
  \begin{equation*}
	\left\{
	\begin{aligned}
	&[u'(r)^{p-1}]'+\frac{n-1}{r}u'(r)^{p-1}=f_1(r)g_1(v(r))u'(r)^{\alpha} &&\quad\mbox{for all}\quad 0<r<R,\\
    &[v'(r)^{p-1}]'+\frac{n-1}{r}v'(r)^{p-1}=f_2(r)g_2(v(r))h(u'(r))&&\quad\mbox{for all}\quad 0<r<R,\\
    &u'(0)=v'(0)=0, u(0)> 0, v(0)>0.
	\end{aligned}
	\right.
\end{equation*}
We can rewrite the first equation above to give us the following system
  \begin{equation}\label{sysrad1}
	\left\{
	\begin{aligned}
	&[u'(r)^{p-1-\alpha}]'+\frac{\delta}{r}u'(r)^{p-1-\alpha}=\frac{\delta}{n-1}f_1(r)g_1(v(r)) &&\quad\mbox{for all}\quad 0<r<R,\\
    &[v'(r)^{p-1}]'+\frac{n-1}{r}v'(r)^{p-1}=f_2(r)g_2(v(r))h(u'(r))&&\quad\mbox{for all}\quad 0<r<R,\\
    &u'(0)=v'(0)=0, u(0)> 0, v(0)>0,
	\end{aligned}
	\right.
\end{equation}
where $\delta$ is defined as
\begin{equation}\label{delta def}
    \delta=\frac{(n-1)(p-1-\alpha)}{p-1}>0.
\end{equation}
Since $u',v'>0$ for all $0<r<R$ from Lemma \ref{monotonic}, the upper bounds in \eqref{convex bounds1} and \eqref{convex bounds2} clearly hold. When it comes to the lower bounds, we first rewrite \eqref{sysrad1} as
 \begin{equation}\label{sysrad2}
	\left\{
	\begin{aligned}
	&[r^{\delta}u'(r)^{p-1-\alpha}]'=\frac{\delta}{n-1}r^{\delta}f_1(r)g_1(v(r)) &&\quad\mbox{for all}\quad 0<r<R,\\
    &[r^{n-1}v'(r)^{p-1}]'=r^{n-1}f_2(r)g_2(v(r))h(u'(r))&&\quad\mbox{for all}\quad 0<r<R,\\
    &u'(0)=v'(0)=0, u(0)> 0, v(0)>0.
	\end{aligned}
	\right.
\end{equation}
Now, since $v(r)$ is increasing for $0<r<R$, recalling our assumption (A1) the first equation in \eqref{sysrad2} tells us
\begin{equation*}
    r^{\delta}u'(r)^{p-1-\alpha}=\frac{\delta}{n-1}\int_{0}^{r}t^{\delta}f_1(t)g_1(v(t))dt
    \leq \frac{\delta}{n-1}f_1(r)g_1(v(r)) \int_{0}^{r}t^{\delta}dt 
\end{equation*}
for all $0<r<R$. So we have the estimate 
\begin{equation}\label{u' estimate}
    \frac{1}{r}u'(r)^{p-1-\alpha}\leq \frac{\delta}{(\delta+1)(n-1)}f_1(r)g_1(v(r)) 
\end{equation}
for all $0<r<R$, which, when combined with the first equation in \eqref{sysrad1} yields
\begin{align*}
    [u'(r)^{p-1-\alpha}]'&=\frac{\delta}{n-1}f_1(r)g_1(v(r))-\frac{\delta}{r}u'(r)^{p-1-\alpha}\\
    &\geq \frac{\delta}{(\delta+1)(n-1)}f_1(r)g_1(v(r)). 
\end{align*}
Recalling the definition of $\delta$ we see that this is precisely the lower bound in \eqref{convex bounds1}.
Now, since we have just shown that $u'(r)$ is increasing, we can take the second equation in \eqref{sysrad2} and argue in a similar way to conclude
\begin{equation*}
      r^{n-1}v'(r)^{p-1}
    \leq f_2(r)g_2(v(r))h(u'(r)) \int_{0}^{r}t^{n-1}dt=\frac{1}{n}r^{n} f_2(r)g_2(v(r))h(u'(r)) , 
\end{equation*}
for all $0<r<R$. One may now take the second equation in \eqref{sysrad1} and argue in an analogous manner to above to obtain the lower bound in \eqref{convex bounds2}. 
\end{proof}

\begin{lemma}\label{no u2inf}
    Assume (A1), $\Omega=B_R$ and $0\leq\alpha<p-1$. Then \eqref{system} does not admit any non-constant positive radial solutions $(u,v)$   satisfying 
    \begin{equation}\label{bad bounds}
        \lim\limits_{r\to R^{-}}u(r)=\infty\quad\mbox{and}\quad \lim\limits_{r\to R^{-}}v(r)<\infty
    \end{equation}
\end{lemma}
\begin{proof}
    Suppose $(u,v)$ was such a solution. For each $0\leq r<R$ we have
    \begin{equation*}
        u(r)=u(0)+\int_{0}^{r}u'(t)dt,
    \end{equation*}
    so from \eqref{bad bounds} we see that $\lim\limits_{r\to R^{-}}u'(r)=\infty$. But from \eqref{u' estimate}, \eqref{bad bounds} and (A1) we also have
    \begin{equation*}
        u'(r)^{p-1-\alpha}\leq \frac{\delta r}{(\delta+1)(n-1)}f_1(r)g_1(v(r))\leq C
    \end{equation*}
    for all $0<r<R$, and some constant $C>0$. Letting $r\to R^{-}$ we clearly reach a contradiction.
\end{proof}

\begin{lemma}\label{h H comp quas}
    Let $h(t)$ satisfy (A1), and $H(t)=\int_{0}^{t}h(s)ds$. Then for $s> 0$ we have
    \begin{equation*}
  (p-1)^{2p-1}\left(\int_{0}^{s}H(t)^{\frac{1}{p-1}}dt\right)^{p-1}\leq (p-1)^{p-1}\left(\int_{0}^{ps}h(t)^{\frac{1}{p}}dt\right)^{p}\leq \left(\int_{0}^{p^2s}H(t)^{\frac{1}{p-1}}dt\right)^{p-1}
\end{equation*}
and consequently, for each $\nu>0$
\begin{equation*}
    \int_{1}^{\infty}\frac{ds}{\left(\int_{0}^{s}H(t)^{\frac{1}{p-1}}dt\right)^{\nu(p-1)}}<\infty\quad \mbox{if and only if} \quad\int_{1}^{\infty}\frac{ds}{\left(\int_{0}^{s}h(t)^{\frac{1}{p}}dt\right)^{\nu p}}<\infty.
\end{equation*}
\end{lemma}
\begin{proof}
We start by noting that since $p>1$, the function $x\mapsto x^p$ is convex. If we fix $s>0$ we have
\begin{align*}
     \left(\int_{0}^{s}h(t)^{\frac{1}{p}}dt\right)^{p}= s^p\left(\frac{1}{s}\int_{0}^{s}h(t)^{\frac{1}{p}}dt\right)^{p}&\leq  s^{p-1}\int_{0}^{s}h(t)dt \quad\mbox{by Jensen's inequality}\\
     &=(p-1)^{1-p}\left(H(s)^{\frac{1}{p-1}}\int_{s}^{ps}d\sigma\right)^{p-1} \\
     &\leq (p-1)^{1-p}\left(\int_{s}^{ps}H(\sigma)^{\frac{1}{p-1}}d\sigma\right)^{p-1}\\
     &\leq (p-1)^{1-p}\left(\int_{0}^{ps}H(\sigma)^{\frac{1}{p-1}}d\sigma\right)^{p-1}
\end{align*}
and the second inequality follows. When it comes to the first inequality, we have
\begin{align*}
     \left(\int_{0}^{s}H(t)^{\frac{1}{p-1}}dt\right)^{p-1}&\leq  \left(H(s)^{\frac{1}{p-1}}\int_{0}^{s}dt\right)^{p-1}\ &&\mbox{since}\ H\ \mbox{is increasing}\\
     &= s^{p-1}\int_{0}^{s}h(t)dt\quad&&\mbox{by the definition of}\ H \\
     &\leq s^{p}h(s)\quad&&\mbox{since}\ h\ \mbox{is increasing}\\
     &= \left(\frac{1}{p-1}h(s)^{\frac{1}{p}}\int_{s}^{ps}dt\right)^p\\
     &\leq \left(\frac{1}{p-1}\int_{0}^{ps}h(t)^{\frac{1}{p}}dt\right)^p
\end{align*}
and the first inequality follows.
\end{proof}

\bigskip

\noindent Lemma \ref{h H comp quas} above is an extension of \cite[Lemma 4.1]{S15}, and is key to proving the sharpness of the conditions in Theorem \ref{quas case}. If we define for each $\theta>0$
\begin{equation}\label{H theta def}
    H_{\theta}(t)\coloneqq\int_{0}^{t}h(s^{\theta})ds,\quad t>0,
\end{equation} 
then Lemma \ref{h H comp quas} tells us that for each $\nu>0$ we have
\begin{equation*}
    \int_{1}^{\infty}\frac{ds}{\left(\int_{0}^{s}H_{\theta}(t)^{\frac{1}{p-1}}dt\right)^{\nu(p-1)}}<\infty\quad \mbox{if and only if} \quad\int_{1}^{\infty}\frac{ds}{\left(\int_{0}^{s}h(t^{\theta})^{\frac{1}{p}}dt\right)^{\nu p}}<\infty.
\end{equation*}
It is this formulation, with $\theta=\frac{1}{p-1-\alpha}$ and $\nu=\frac{k_1}{k_1p+p-1-k_2}$, that we shall use in the subsequent sections.

\section{ Proof of Theorem \ref{quas case}  }\label{proof thm quas}
Before the main proof, we first argue that system \eqref{system} has a non-constant positive radial solution in $\Omega=B_R$ for any $R>0$. To see this, we note the existence of a non-constant positive radial solution of \eqref{system} in a small ball $B_{\rho}$ follows from a standard fixed point argument (see for example \cite[Proposition A1]{FV17} for details). In particular, the map
\begin{equation*}
    \mathcal{T}:C^{1}[0,\rho]\times C^{1}[0,\rho]\rightarrow C^{1}[0,\rho]\times C^{1}[0,\rho]\
\end{equation*}
given by
\begin{equation*}
    \mathcal{T}[u,v](r)= \begin{pmatrix}\mathcal{T}_1[u,v](r)\\\mathcal{T}_2[u,v](r)\end{pmatrix}
\end{equation*}
where 
\begin{equation*}
\left\{
\begin{aligned}
\mathcal{T}_1[u,v](r)&=u(0)+\int_{0}^{r}\left[\frac{\delta}{n-1}s^{-\delta}\int_{0}^{s}\tau^{\delta}f_1(\tau)g_1(v(\tau))d\tau\right]^{\frac{1}{p-1-\alpha}} \\
\mathcal{T}_2[u,v](r)&=v(0)+\int_{0}^{r}\left[s^{1-n}\int_{0}^{s}\tau^{n-1}f_2(\tau)g_2(v(\tau))h(|u'(\tau)|)d\tau\right]^{\frac{1}{p-1}}
\end{aligned}
\right.
\end{equation*}
and $\delta$ is defined as in \eqref{delta def}, has a fixed point $(u,v)$ for sufficiently small $\rho>0$. Note from Lemma \ref{alpha<p-1} that $p-1-\alpha>0$. Now, let $R>0$ be arbitrary. We claim that \eqref{system} has a positive radial solution in $B_R$. By the above the following system
\begin{equation*}
\left\{
\begin{aligned}
\Delta_{p} \Tilde{u}&=\Tilde{f}_1(|x|)\Tilde{g}_1(\Tilde{v})|\nabla \Tilde{u}|^{\alpha}   &&\quad\mbox{ in } B_{\rho}, \\
\Delta_{p} \Tilde{v}&=\Tilde{f}_2(|x|)\Tilde{g}_2(\Tilde{v})\Tilde{h}(|\nabla \Tilde{u}|) &&\quad\mbox{ in } B_{\rho},
\end{aligned}
\right.
\end{equation*}
has a positive radial solution $(\Tilde{u},\Tilde{v})$ for sufficiently small $\rho$, where
\begin{align*}
    \Tilde{f}_1(r)&=\lambda^{p-1-\alpha}f_1(\lambda r)\quad \Tilde{f}_2(r)=\lambda^{p-1}f_2(\lambda r)\\
    \Tilde{g}_1(t)&=g_1(t)\qquad\qquad\ \ \Tilde{g}_2(t)=g_2(t)\\
    \Tilde{h}(t)&=h\left(\frac{t}{\lambda}\right),
\end{align*}
for $\lambda>0$. We then have that
\begin{equation*}
    u(r)=\Tilde{u}\left(\frac{r}{\lambda}\right)\quad v(r)=\Tilde{v}\left(\frac{r}{\lambda}\right)
\end{equation*}
is a solution of \eqref{system} in $B_{\lambda\rho}$, and the claim follows by letting $\lambda=\frac{R}{\rho}$. We are thus guaranteed the existence of non-constant positive radial solutions of \eqref{system} in $B_R$, and are now in a position to classify them according to their behaviour as $r\to R^{-}$. In what follows, $C$ will be used to denote a positive constant which may vary on each occurrence. In general $C=C(n,p,\alpha,R,u,v)$, but its exact value will not be important. 
\bigskip

 \noindent When it comes to part 1, we first assume that \eqref{system} admits a solution $(u,v)$ satisfying
\begin{equation}\label{vtoinf}
    \lim\limits_{r\to R^{-}}v(r)=\infty.
\end{equation}
From assumption (A2), for some constants $c_1,c_2>0$ we have the estimates 
 \begin{align*}
    c_1v(r)^{k_1}\leq g_1(v(r))&\leq \frac{g_1(v(0))}{v(0)^{k_1}}v(r)^{k_1}\\
    c_2v(r)^{k_2}\leq g_2(v(r))&\leq \frac{g_2(v(0))}{v(0)^{k_2}}v(r)^{k_2}
\end{align*}
for all $r>0$. For simplicity we shall let $w=u'$, and we note that $w$ is positive and increasing on $(0,R)$ by Lemmas \ref{monotonic} and \ref{convex}. Multiplying the right inequality in \eqref{convex bounds2} by $v'$, which is positive by Lemma \ref{monotonic}, and integrating over $[0,r]$ for $r<R$ we find
\begin{equation*}
 \int_{0}^{r}[v'(t)^{p-1}]'v'(t)dt \leq \int_{0}^{r} f_2(t)g_2(v(t))h(w(t))v'(t) dt. 
\end{equation*}
Now, by Lemmas \ref{monotonic} and \ref{convex}, we know that $v(r)$ and $w(r)$ are increasing in $(0,R)$, so by our assumptions (A1)-(A2) the above inequality becomes 
\begin{equation*}
   \frac{p-1}{p}\int_{0}^{r} (v'(t)^{p})'dt\leq Ch(w(r))\int_{0}^{r}v(t)^{k_2}v'(t)dt 
\end{equation*}
from which we have, recalling $v'(0)=0$,
\begin{equation*}\label{root f bound}
v'(r)^{p}\leq Ch(w(r))\int_{v(0)}^{v(r)}t^{k_2}dt\leq Cv(r)^{k_2+1}h(w(r))
\end{equation*}
for all $ r\in(0,R)$. In other words,
\begin{equation}\label{v'G2}
    v'(r)v(r)^{-\frac{k_2+1}{p}}\leq Ch(w(r))^{\frac{1}{p}}
\end{equation}
for all $ r\in(0,R)$. 
Fix now some $\rho\in(0,R)$. Inequality \eqref{convex bounds1} and our assumptions (A1)-(A2) tell us that for all  $ r\in[\rho,R)$ we have
\begin{equation*}
    Cv(r)^{k_1}\leq f_1(\rho)g_1(v(r))\leq f_1(r)g_1(v(r))\leq \frac{n(p-1-\alpha)+\alpha}{p-1-\alpha}[w(r)^{p-1-\alpha}]'.
\end{equation*}
Combining this with \eqref{v'G2} yields
\begin{equation*}
     v'(r)v(r)^{k_1-\frac{k_2+1}{p}}\leq Ch(w(r))^{\frac{1}{p}}[w(r)^{p-1-\alpha}]',
\end{equation*}
which we can integrate over $[\rho,r]$ for any $r\in(\rho,R)$ to obtain, after changing variables,
\begin{align}
     \int_{v(\rho)}^{v(r)}t^{k_1-\frac{k_2+1}{p}}dt&\leq C\int_{w(\rho)^{p-1-\alpha}}^{w(r)^{p-1-\alpha}}h(t^{\frac{1}{p-1-\alpha}})^{\frac{1}{p}}dt\nonumber \\
     &\leq C\int_{0}^{w(r)^{p-1-\alpha}}h(t^{\frac{1}{p-1-\alpha}})^{\frac{1}{p}}dt\label{Gv-Gv0}
\end{align}
for all $ r\in(\rho,R)$. Noting that $k_1p> k_1\geq k_2$ by our assumption (A2), we may rewrite \eqref{Gv-Gv0} as
\begin{align*}
    v(r)^{\frac{k_1p+p-1-k_2}{p}}&\leq v(\rho)^{\frac{k_1p+p-1-k_2}{p}}+ C\int_{0}^{w(r)^{p-1-\alpha}}h(t^{\frac{1}{p-1-\alpha}})^{\frac{1}{p}}dt\\
    &=\left(\frac{v(\rho)^{\frac{k_1p+p-1-k_2}{p}}}{\int_{0}^{w(r)^{p-1-\alpha}}h(t^{\frac{1}{p-1-\alpha}})^{\frac{1}{p}}dt}+C\right)\int_{0}^{w(r)^{p-1-\alpha}}h(t^{\frac{1}{p-1-\alpha}})^{\frac{1}{p}}dt\\
    &\leq \left(\frac{v(\rho)^{\frac{k_1p+p-1-k_2}{p}}}{\int_{0}^{w(\rho)^{p-1-\alpha}}h(t^{\frac{1}{p-1-\alpha}})^{\frac{1}{p}}dt}+C\right)\int_{0}^{w(r)^{p-1-\alpha}}h(t^{\frac{1}{p-1-\alpha}})^{\frac{1}{p}}dt
\end{align*}
and since $\rho>0$ is fixed, we can express the above as
\begin{equation*}\label{G less h int}
    v(r)^{\frac{k_1p+p-1-k_2}{p}}\leq C\int_{0}^{w(r)^{p-1-\alpha}}h(t^{\frac{1}{p-1-\alpha}})^{\frac{1}{p}}dt,
\end{equation*}
for all $r\in[\rho,R)$. In other words
\begin{equation*}
    v(r)\leq C\left(\int_{0}^{w(r)^{p-1-\alpha}}h(t^{\frac{1}{p-1-\alpha}})^{\frac{1}{p}}dt\right)^{\frac{p}{k_1p+p-1-k_2}}
\end{equation*}
for all $r\in[\rho,R)$. Combining this with \eqref{convex bounds1} we have
\begin{align}\label{g1 upp bound}
    [w(r)^{p-1-\alpha}]'&\leq \frac{p-1-\alpha}{p-1}f_1(r)g_1(v(r))\nonumber\\
    &\leq Cv(r)^{k_1}\nonumber\\
    &\leq C\left(\int_{0}^{w(r)^{p-1-\alpha}}h(t^{\frac{1}{p-1-\alpha}})^{\frac{1}{p}}dt\right)^{\frac{k_1p}{k_1p+p-1-k_2}}
\end{align}
for all $r\in[\rho,R)$.
It follows from \eqref{g1 upp bound} that for some constant $C>0$ and all $r\in[\rho,R)$
\begin{equation*}
    \frac{[w(r)^{p-1-\alpha}]'}{\left(\int_{0}^{w(r)^{p-1-\alpha}}h(t^{\frac{1}{p-1-\alpha}})^{\frac{1}{p}}dt\right)^{\frac{k_1p}{k_1p+p-1-k_2}}}\leq C. 
\end{equation*}
 We can integrate both sides of this expression over $[\rho,r]$ for $r\in(\rho,R)$ to find
\begin{equation*}
     \int_{\rho}^{r}\frac{[w(s)^{p-1-\alpha}]'\ ds}{\left(\int_{0}^{w(s)^{p-1-\alpha}}h(t^{\frac{1}{p-1-\alpha}})^{\frac{1}{p}}dt\right)^{\frac{k_1p}{k_1p+p-1-k_2}}}\leq C \int_{\rho}^{r}ds
\end{equation*}
or, after a change of variables
\begin{equation}\label{bounded above int}
    \int_{w(\rho)^{p-1-\alpha}}^{w(r)^{p-1-\alpha}}\frac{ds}{\left(\int_{0}^{s}h(t^{\frac{1}{p-1-\alpha}})^{\frac{1}{p}}dt\right)^{\frac{k_1p}{k_1p+p-1-k_2}}}\leq C(r-\rho)\leq Cr.
\end{equation}
Letting $r\to R^{-}$, we see, recalling \eqref{convex bounds1} and \eqref{vtoinf}
\begin{equation*}
     \int_{w(\rho)^{p-1-\alpha}}^{\infty}\frac{ds}{\left(\int_{0}^{s}h(t^{\frac{1}{p-1-\alpha}})^{\frac{1}{p}}dt\right)^{\frac{k_1p}{k_1p+p-1-k_2}}}\leq CR<\infty.
\end{equation*}
Hence,
\begin{equation*}\label{int sqrt h}
   \int_{1}^{\infty}\frac{ds}{\left(\int_{0}^{s}h(t^{\frac{1}{p-1-\alpha}})^{\frac{1}{p}}dt\right)^{\frac{k_1p}{k_1p+p-1-k_2}}}<\infty.
\end{equation*}
Recalling Lemma \ref{no u2inf}, we thus have that all positive radial solutions of \eqref{system} satisfy (B1) if \eqref{u v bound} holds. In order to prove the converse statement, we first note that system \eqref{system} has a positive radial solution $(u,v)$ defined on some maximal interval $[0,R_{0})$. Suppose then that \eqref{u v bound} does not hold, and fix $\rho\in(0,R_0)$. Recall  from \eqref{convex bounds1},\eqref{convex bounds2} and our assumptions (A1)-(A2) that we have 
    \begin{align}
        [w(r)^{p-1-\alpha}]'&\leq Cv(r)^{k_1}\label{vk1}\\
        v(r)^{k_2}h(w(r))&\leq C[v'(r)^{p-1}]'\label{vk2}
    \end{align}
    for all $ r\in[\rho,R_0)$. 
    Multiplying \eqref{vk1} and \eqref{vk2} and rearranging gives
    \begin{equation*}
        h(w(r))  [w(r)^{p-1-\alpha}]'\leq Cv(r)^{k_1-k_2}[v'(r)^{p-1}]'
    \end{equation*}
    for all $ r\in[\rho,R_0)$. Integrating the above inequality over $[\rho,r]$ for any $r\in(\rho,R_0)$ we obtain
    \begin{align*}
        \int_{w(\rho)^{p-1-\alpha}}^{w(r)^{p-1-\alpha}}h(t^{\frac{1}{p-1-\alpha}})dt&\leq C\int_{\rho}^{r}v(t)^{k_1-k_2}[v'(t)^{p-1}]'dt\\
        &\leq C\int_{0}^{r} v(t)^{k_1-k_2}[v'(t)^{p-1}]'dt\\
        &\leq Cv(r)^{k_1-k_2}\int_{0}^{r} [v'(t)^{p-1}]'dt\\
        &\leq Cv(r)^{k_1-k_2}v'(r)^{p-1},
    \end{align*}
recalling that $k_1\geq k_2$ by (A2) and $v'(0)=0$. Taking the extreme left and right above, one has
\begin{align}\label{wvv'}
    \int_{0}^{w(r)^{p-1-\alpha}}h(t^{\frac{1}{p-1-\alpha}})dt&\leq \int_{0}^{w(\rho)^{p-1-\alpha}}h(t^{\frac{1}{p-1-\alpha}})dt +Cv(r)^{k_1-k_2}v'(r)^{p-1}\nonumber\\ 
    &=\left(\frac{\int_{0}^{w(\rho)^{p-1-\alpha}}h(t^{\frac{1}{p-1-\alpha}})dt}{v(r)^{k_1-k_2}v'(r)^{p-1}}+C\right)v(r)^{k_1-k_2}v'(r)^{p-1}\nonumber\\
    &\leq \left(\frac{\int_{0}^{w(\rho)^{p-1-\alpha}}h(t^{\frac{1}{p-1-\alpha}})dt}{v(\rho)^{k_1-k_2}v'(\rho)^{p-1}}+C\right)v(r)^{k_1-k_2}v'(r)^{p-1}\nonumber\\
    &\leq  Cv(r)^{k_1-k_2}v'(r)^{p-1}
\end{align}
for all $r\in[\rho,R_0)$, recalling that $v'(r)$ is positive for $r>0$ by Lemma \ref{monotonic}. We now let $\theta=\frac{1}{p-1-\alpha}$, and recall from \eqref{H theta def} that we define
    \begin{equation*}
        H_{\theta}(t)=\int_{0}^{t}h(s^{\theta})ds,\ t>0.
    \end{equation*}
Multiplying inequalities \eqref{vk1} and \eqref{wvv'} thus gives
\begin{equation*}
    H_{\theta}(w(r)^{p-1-\alpha})^{\frac{1}{p-1}}[w(r)^{p-1-\alpha}]'\leq Cv(r)^{\frac{k_1p-k_2}{p-1}}v'(r)
\end{equation*}
for all $ r\in[\rho,R_0)$. A further integration over $[\rho,r]$ yields
\begin{equation*}
    \int_{w(\rho)^{p-1-\alpha}}^{w(r)^{p-1-\alpha}}H_{\theta}(t)^{\frac{1}{p-1}}dt \leq Cv(r)^{\frac{k_1p+p-1-k_2}{p-1}},
\end{equation*}
for all $r\in(\rho,R_0)$. Similar to above, since $k_1p+p-1-k_2>0$, there exists $C>0$ such that
\begin{equation*}
    \int_{0}^{w(r)^{p-1-\alpha}}H_{\theta}(t)^{\frac{1}{p-1}}dt \leq Cv(r)^{\frac{k_1p+p-1-k_2}{p-1}}
\end{equation*}
for all $r\in(\rho,R_0)$. Using \eqref{convex bounds1} and our assumption (A2) 
the above inequality becomes
\begin{equation*}
    \int_{0}^{w(r)^{p-1-\alpha}}H_{\theta}(t)^{\frac{1}{p-1}}dt \leq C[v(r)^{k_1}]^{\frac{k_1p+p-1-k_2}{k_1(p-1)}}\leq C([w(r)^{p-1-\alpha}]')^{\frac{k_1p+p-1-k_2}{k_1(p-1)}}
\end{equation*}
for all $r\in(\rho,R_0)$, yielding
\begin{equation}\label{clessw}
    C\leq \frac{[w(r)^{p-1-\alpha}]'}{\left(\int_{0}^{w(r)^{p-1-\alpha}}H_{\theta}(t)^{\frac{1}{p-1}}dt\right)^{\frac{k_1(p-1)}{k_1p+p-1-k_2}}}
.
\end{equation}
Integrating \eqref{clessw} over $[\rho,r]$ and letting $r\to R_0^{-}$ we find
\begin{align}\label{intwbounded}
    C(R_0-\rho)&\leq \int_{w(\rho)^{p-1-\alpha}}^{w(R_0^{-})^{p-1-\alpha}}\frac{ds}{\left(\int_{0}^{s}H_{\theta}(t)^{\frac{1}{p-1}}dt\right)^{\frac{k_1(p-1)}{k_1p+p-1-k_2}}} \nonumber\\
    &\leq \int_{w(\rho)^{p-1-\alpha}}^{\infty}\frac{ds}{\left(\int_{0}^{s}H_{\theta}(t)^{\frac{1}{p-1}}dt\right)^{\frac{k_1(p-1)}{k_1p+p-1-k_2}}}
\end{align}
 where $w(R_0^{-})\in \bR^{+}\cup\{\infty\}$ is understood to mean $\lim_{r\to R_0^{-}}w(r)$. If \eqref{u v bound} does not hold, then Lemma \ref{h H comp quas} together with \eqref{intwbounded} implies the maximal interval of existence $[0,R_0)$ is finite. By Lemmas \ref{monotonic} and \ref{no u2inf}, this implies $v(R_0^{-})=\infty$, proving part 1. 
 
 Turning now to part 2, assume that \eqref{vtoinf} holds and let $\Phi:(0,\infty)\rightarrow (0,\infty)$ be defined as
\begin{equation*}
  \Phi(t)\coloneqq \int_{t}^{\infty}\frac{ds}{\left(\int_{0}^{s}h(t^{\frac{1}{p-1-\alpha}})^{\frac{1}{p}}dt\right)^{\frac{k_1p}{k_1p+p-1-k_2}}}.
\end{equation*}
 We note that $\Phi$ is decreasing and by part 1 we have $\lim\limits_{t\to \infty}\Phi(t)=0$. From \eqref{bounded above int} and \eqref{intwbounded} we see there exists $\rho\in(0,R)$ such that
\begin{equation*}
     C_1(R-r)\leq \Phi(w(r)^{p-1-\alpha})\leq C_2(R-r)
\end{equation*}
for all $ r\in[\rho,R)$ and some constants $C_2>C_1>0$. Since $\Phi$ is decreasing, this implies
\begin{equation*}
    \Phi^{-1}(C_2(R-r))^{\frac{1}{p-1-\alpha}}\leq w(r)\leq \Phi^{-1}(C_1(R-r))^{\frac{1}{p-1-\alpha}} .
\end{equation*}
Now, recalling that for all $ r\in[\rho,R)$
\begin{equation*}
u(r)=u(\rho)+\int_{\rho}^{r}w(t)dt,
\end{equation*}
we see that $\lim\limits_{r\to R^{-}}u(r)<\infty$ if and only if
\begin{equation*}\label{int w finite}
    \int_{\rho}^{R}w(t)dt<\infty,
\end{equation*}
which holds if and only if
\begin{equation*}
    \int_{\rho}^{R}\Phi^{-1}(C(R-t))^{\frac{1}{p-1-\alpha}}dt<\infty
\end{equation*}
for some $C>0$. Hence if $\lim\limits_{r\to R^{-}}u(r)<\infty$, after a change of variables we see
\begin{equation*}
    \int_{0}^{C(R-\rho)}\Phi^{-1}(s)^{\frac{1}{p-1-\alpha}}ds<\infty,
\end{equation*}
which yields
\begin{equation*}
    \int_{0}^{1}\Phi^{-1}(s)^{\frac{1}{p-1-\alpha}}ds<\infty.
\end{equation*}
The change of variables $t=\Phi^{-1}(s)$ gives us that $\lim\limits_{r\to R^{-}}u(r)<\infty$ if and only if
\begin{equation*}
 \int_{1}^{\infty}\frac{s^{\frac{1}{p-1-\alpha}}\ ds}{\left(\resizebox{\width}{1.2\height}{$\int_{0}^{s}$}{h(t^{\frac{1}{p-1-\alpha}})^{\frac{1}{p}}}dt\right)^{\frac{k_1p}{k_1p+p-1-k_2}}}<\infty,
\end{equation*}
finishing the proof. $\hfill\blacksquare$

\begin{remark}
\normalfont{
A possible extension of the above results would be to somehow relax assumption (A2), however the analysis becomes quite technical in this case, and the results obtained are no longer sharp, so we do not discuss it here. 
}
\end{remark}

\section{Proof of Theorem \ref{global}}\label{proof thm pol}

\noindent First assume that \eqref{u v bound} holds and $0\leq \alpha<p-1$. From the discussion at the beginning of Section \ref{proof thm quas}, we know that when $0\leq \alpha<p-1$ we are able to construct a non-constant positive radial solution in a maximal ball. By Lemma \ref{monotonic}, both $u$ and $v$ are increasing, and by Lemma \ref{no u2inf} and Theorem \ref{quas case} part 1, we know that $u$ and $v$ are bounded when \eqref{u v bound} holds. Hence the domain of existence must be all of $\bR^{n}$.

Conversely, we know from Lemma \ref{alpha<p-1} that if $\alpha\geq p-1$, then no positive radial solution of \eqref{system} exists in $B_R$ for any $R>0$. So suppose now that $0\leq\alpha<p-1$, \eqref{u v bound} does not hold, and $(U,V)$ is a global positive radial solution of \eqref{system}. We consider the related system
\begin{equation}\label{aux system}
\left\{
\begin{aligned}
\Delta_{p} u&=\Tilde{f}_1(|x|)\Tilde{g}_1(v)|\nabla u|^{\alpha}   &&\quad\mbox{ in } \Omega, \\
\Delta_{p} v&=\Tilde{f}_2(|x|)\Tilde{g}_2(v)\Tilde{h}(|\nabla u|) &&\quad\mbox{ in } \Omega,
\end{aligned}
\right.
\end{equation}
where
\begin{align*}
    \Tilde{f}_1(r)&=\lambda^{p-1-\alpha}f_1(\lambda r)\quad \Tilde{f}_2(r)=\lambda^{(p-1)(1-c)}f_2(\lambda r)\\
    \Tilde{g}_1(t)&=g_1(\lambda^{c}t)\qquad\qquad\ \ \Tilde{g}_2(t)=g_2(\lambda^{c}t)\\
    \Tilde{h}(t)&=h\left(\frac{t}{\lambda}\right)
\end{align*}
for some $0<c<1$ and $\lambda>0$. Then $(U_\lambda(r),V_{\lambda}(r))=(U(\lambda r),\lambda^{-c}V(\lambda r))$ is a global positive radial solution of \eqref{aux system}. The system \eqref{aux system} also satisfies the hypotheses (A1)-(A2), so by Theorem \ref{quas case}, since \eqref{u v bound} does not hold, there exists a positive radial solution $(u,v)$ of \eqref{aux system} satisfying, without loss of generality, $\lim_{r\to 1^{-}}v(r)=\infty$. By letting $\lambda>0$ be sufficiently small, we may assume $V_{\lambda}(0)>v(0)>0$. Let $(0,R_0)\subset (0,1)$ be the maximal interval on which
\begin{equation} \label{comp2}
 V_{\lambda}(r) > v(r) \qquad\text{for all $0<r<R_0$.}
\end{equation}
Now, since $(U_{\lambda},V_{\lambda})$ and $(u,v)$ both solve system \eqref{aux system}, if we rewrite system \eqref{aux system} in an analogous manner to \eqref{sysrad2}, we see that $(U_{\lambda},V_{\lambda})$ and $(u,v)$ also satisfy 
\begin{equation*}
\left[ r^\delta (U_{\lambda}'(r)^{p-1-\alpha} - u'(r)^{p-1-\alpha}) \right]' =
\frac{\delta}{n-1} \,r^\delta \Tilde{f}_1(r) \cdot \left[\Tilde{g}_1(V_{\lambda}) - \Tilde{g}_1(v)\right],
\end{equation*}
with $\delta$ as defined in \eqref{delta def}.
Along the interval $(0,R_0)$, we have $V_{\lambda}>v>0$ by \eqref{comp2}, so $\Tilde{g}_1(V_{\lambda}) > \Tilde{g}_1(v)$
by (A2).  Since $\Tilde{f}_1(r)$ is positive by (A1), the right hand side above is clearly positive, so it follows that $U_{\lambda}'(r)>u'(r)$ for all $0<r<R_0$, recalling that $p-1-\alpha>0$.

When it comes to the functions $V_{\lambda},v$,  by using the second equation in \eqref{aux system}, we get
\begin{equation}\label{vlv comp}
\left[r^{n-1} (V_{\lambda}'(r)^{p-1} - v'(r)^{p-1})\right]' =
r^{n-1} \Tilde{f}_2(r) \cdot \left[\Tilde{g}_2(V_{\lambda}) \Tilde{h}(U_{\lambda}') - \Tilde{g}_2(v) \Tilde{h}(u')\right].
\end{equation}
Since $U_{\lambda}'>u'$ on the interval $(0,R_0)$ by above, we have $\Tilde{h}(U_{\lambda}') \geq \Tilde{h}(u')$ on $(0,R_0)$ by
(A1). In addition, $V_{\lambda}>v$ on $(0,R_0)$ by \eqref{comp2}, and thus $\Tilde{g}_2(V_{\lambda}) \geq \Tilde{g}_2(v)$ on $(0,R_0)$ by (A1), showing that the right hand side of \eqref{vlv comp} is non-negative. If the right hand side of \eqref{vlv comp} is identically zero, then since $V_{\lambda}'(0)=v'(0)=0$, we in fact have $V_{\lambda}'(r)=v'(r)$ on $(0,R_0)$. If the right hand side of \eqref{vlv comp} becomes positive on $(0,R_0)$, we may argue similarly to above to conclude $V_{\lambda}'(r)\geq v'(r)$ for all $0<r<R_0$. So in both cases we have $V_{\lambda}'(r)\geq v'(r)$ for all $0<r<R_0$.
But since $V_{\lambda}(r)> v(r)$ on $(0,R_0)$ it cannot possibly be the case that $v(r)\geq V_{\lambda}(r)$ at $r=R_0$.  In other words,
the maximal interval on which \eqref{comp2} holds coincides with the maximal interval on which $v(r)$ is defined, implying $\lim_{r\to 1^{-}}V_{\lambda}(r)\geq\lim_{r\to 1^{-}}v(r)=\infty$, which contradicts the fact that $(U_{\lambda},V_{\lambda})$ is a global solution of \eqref{aux system}.$\hfill \blacksquare$

\section*{Acknowledgements}
The author thanks Paschalis Karageorgis and Gurpreet Singh for their helpful discussions and comments during the preparation of this paper.

\section*{Funding}
The author acknowledges the financial support of The Irish Research Council Postgraduate Scholarship, under grant number GOIPG/2022/469.

\bibliography{existence}
\bibliographystyle{plain}
\end{document}